\documentclass[12pt]{amsart}%
\usepackage{amscd,amsmath,latexsym,amsthm,amsfonts,amssymb,graphicx,color,geometry,hyperref}
\usepackage[utf8]{inputenc}
\usepackage{amsmath}
\usepackage{amsfonts}
\usepackage{amssymb}
\usepackage{graphicx}%
\providecommand{\U}[1]{\protect\rule{.1in}{.1in}}
\providecommand{\U}[1]{\protect\rule{.1in}{.1in}}

\newtheorem{theorem}{Theorem}[section]

\numberwithin{equation}{section}

\begin{document}
\title[Hardy--Littlewood inequalities for multipolynomials]{Hardy--Littlewood inequalities for multipolynomials}
\author[Daniel Tomaz]{Daniel Tomaz}
\address{Departamento de Matem\'{a}tica \\
\indent
	Universidade Federal da Para\'{\i}ba \\
\indent
	58.051-900 - Jo\~{a}o Pessoa, Brazil}
\email{danieltomazmatufpb@gmail.com}
\thanks{The author is supported by Capes}

\begin{abstract}
The notion of multipolynomials was recently introduced and explored by T.
Velanga in \cite{Velanga} as an attempt to encompass the theories of polynomials and multilinear operators. In the present paper we push this subject further, by proving Hardy--Littlewood inequalities for multipolynomials and, \textit{en passant}, a variant
of the Kahane--Salem--Zygmund inequality in this framework.

\end{abstract}
\maketitle



\section{Introduction}

Throughout this paper, $E_{1},\ldots,E_{m},F$ will be Banach spaces over the
scalar field $\mathbb{K}$, which can be either $\mathbb{R}$ or $\mathbb{C}$.
We recall that a map $P:E\longrightarrow F$ is called $m$-homogeneous
polynomial when there is an $m$-linear operator%
\[
T:E\times\cdots\times E\longrightarrow F
\]
such that%
\[
P\left(  x\right)  :=T\left(  x,\ldots,x\right)
\]
for all $x\in E$. As usual, the norm of $P$ is defined by
\[
\left\Vert P\right\Vert :=\sup_{\left\Vert x\right\Vert \leq1}\left\Vert
P\left(  x\right)  \right\Vert .
\]
Given $\alpha=\left(  \alpha_{1},\ldots,\alpha_{n}\right)  \in\mathbb{N}^{n},$
we also define $\left\vert \alpha\right\vert :=\alpha_{1}+\cdots+\alpha_{n}$
and $\mathbf{x}^{\alpha}$ stands for the monomial $x_{1}^{\alpha_{1}}\cdots
x_{n}^{\alpha_{n}}$ for $\mathbf{x=}\left(  x_{1},\ldots,x_{n}\right)
\in\mathbb{K}^{n}$. Thereby, observe that an $m$-homogeneous polynomial in
$\mathbb{K}^{n}$ can be written as%
\[
P\left(  \mathbf{x}\right)  =%
{\textstyle\sum\limits_{\left\vert \alpha\right\vert =m}}
a_{\alpha}(P)\mathbf{x}^{\alpha},
\]
with $a_{\alpha}(P)\in\mathbb{K}$. Recently, T. Velanga (\cite{Velanga}; see also \cite{boo} for a related result)
introduced the notion of multipolynomials between Banach spaces. A map
$P:E_{1}\times\cdots\times E_{m}\longrightarrow F$ is called $\left(
n_{1},\ldots,n_{m}\right)  $-homogeneous polynomial if, for each
$i=1,\ldots,m,$ the mapping $P\left(  x_{1},\ldots,x_{i-1},\cdot
,x_{i+1},\ldots,x_{m}\right)  :E_{i}\longrightarrow F$ is an $n_{i}%
$-homogeneous polynomial for every $x_{1}\in E_{1},\ldots,x_{i-1}\in E_{i-1},$
$x_{i+1}\in E_{i+1},\ldots,x_{m}\in E_{m}$.

Given an integer $m\geq1$, the multilinear Hardy-Littlewood inequality for $2m\leq p\leq\infty$ asserts that there exists a constant $C_{\mathbb{K}%
\text{,}m,p}^{mult}\geq1$ such that, for all continuous $m$-linear forms
$T:\ell_{p}^{n}\times\cdots\times\ell_{p}^{n}\longrightarrow\mathbb{K}$ and
all positive integers $n,$%
\[
\left(
{\textstyle\sum\limits_{j_{1},\ldots,j_{m}=1}^{n}}
\left\vert T\left(  e_{j_{1}},\ldots,e_{j_{m}}\right)  \right\vert
^{\frac{2mp}{mp+p-2m}}\right)  ^{\frac{mp+p-2m}{2mp}}\leq C_{\mathbb{K}%
\text{,}m,p}^{mult}\text{ }\left\Vert T\right\Vert ,
\]
where%
\[
\left\Vert T\right\Vert :=\sup_{z^{\left(  1\right)  },\ldots,z^{\left(
n\right)  }\in B_{\ell_{p}^{n}}}\left\vert T\left(  z^{\left(  1\right)
},\ldots,z^{\left(  n\right)  }\right)  \right\vert .
\]
Moreover, the exponents $\frac{2mp}{mp+p-2m}$ are optimal. When $p=\infty$ we
recover the classical Bohnenblust-Hille inequality (see \cite{bh}). The
polynomial version of this inequality reads as follows: given $m,n\geq1,$ if
$P$ is a $m$-homogeneous polynomial on $\ell_{p}^{n}$ with $2m\leq p\leq
\infty$ given by $P\left(  \mathbf{x}\right)  =%
{\textstyle\sum\limits_{\left\vert \alpha\right\vert =m}}
a_{\alpha}(P)\mathbf{x}^{\alpha},$ then there exists a constant $C_{\mathbb{K}%
\text{,}m,p}^{pol}\geq1$ such that
\begin{equation}
\left(
{\textstyle\sum\limits_{\left\vert \alpha\right\vert =m}} \label{phl}
\left\vert a_{\alpha}\left(  P\right)  \right\vert ^{\frac{2mp}{mp+p-2m}%
}\right)  ^{\frac{mp+p-2m}{2mp}}\leq C_{\mathbb{K}\text{,}m,p}^{pol}\text{
}\left\Vert P\right\Vert .
\end{equation}
 The Hardy-Littlewood inequality for $m<p\leq2m$ (see \cite{Dimant}) asserts that there exists a constant $C_{\mathbb{K}%
 	\text{,}m,p}^{mult}\geq1$ such that for all continuous $m$-linear forms
 $T:\ell_{p}\times\cdots\times\ell_{p}\longrightarrow\mathbb{K}$ we have%
 \begin{equation}
 \left(
 {\displaystyle\sum\limits_{j_{1},\ldots,j_{m}=1}^{\infty}}
 \left\vert T\left(  e_{j_{1}},\ldots,e_{j_{m}}\right)  \right\vert ^{\frac
 	{P}{p-m}}\right)  ^{\frac{p-m}{p}}\leq C_{\mathbb{K}\text{,}m,p}%
 ^{mult}\left\Vert T\right\Vert.
 \end{equation}
 As in the previous case, there exists a polynomial version of this inequality, which asserts that, given an positive integer $m\geq2$ and $m<p\leq2m$, there exists a constant $C_{\mathbb{K}%
 	\text{,}m,p}^{pol}\geq1$ such that 
 \begin{equation}
 \left(
 {\textstyle\sum\limits_{\left\vert \alpha\right\vert =m}} \label{hlpo}
 \left\vert a_{\alpha}\left(  P\right)  \right\vert ^{\frac{p}{p-m}}\right)
 ^{\frac{p-m}{p}}\leq C_{\mathbb{K}\text{,}m,p}^{pol}\left\Vert P\right\Vert .	
 \end{equation} 
 for all $m$-homogeneous polynomial $P:\ell_{p}\longrightarrow\mathbb{K}.$ All the exponents involved are sharp and in \cite{qqq} the difference between the optimal exponents $2mp/(mp+p-2m)$ and $p/(p-m)$ is clearly justified; for recent developments in this field we refer to \cite{araujo, nunes, pt} and the references therein.

The main goal of this paper is to prove a Hardy--Littlewood inequality for multipolynomials.

\section{Preliminary results}

Let us begin by recalling the Kahane-Salem-Zygmund inequality (see
\cite{Albuquerque}):

\begin{theorem}
[Kahane-Salem-Zygmund inequality]\label{ksz}Let $m,n\geq1,$ $p\in\left[
1,+\infty\right]  ^{m}$, and, for  $q\geq1$, define
\[
\alpha\left(  q\right)  :=\left\{
\begin{array}
[c]{c}%
\frac{1}{2}-\frac{1}{q}\text{, if }q\geq2\\
0,\text{ otherwise}%
\end{array}
\right.  .
\]
Then there exists a $m$-linear map $A:\ell_{p}^{n}\times\cdots\times\ell
_{p}^{n}\longrightarrow\mathbb{K}$ of the form
\[
A\left(  x^{\left(  1\right)  },\ldots,x^{\left(  m\right)  }\right)  =%
{\textstyle\sum\limits_{i_{1},\ldots,i_{m}=1}^{n}}
\delta_{i_{1}...i_{m}}x_{i_{1}}^{\left(  1\right)  }\cdots x_{i_{m}}^{\left(
m\right)  },
\]
with $\delta_{i_{1}...i_{m}}\in\{-1,1\},$ such that
\[
\left\Vert A\right\Vert \leq C_{m}.n^{\frac{1}{2}+m\alpha\left(  p\right)  }.
\]

\end{theorem}

From results essentially found in (\cite{Bayart}), we can observe that for
$1\leq q\leq2$ the exponent in the above inequality can be improved.

In fact, let $T:\ell_{2}^{n}\times\cdots\times\ell_{2}^{n}\longrightarrow
\mathbb{K}$ be given by Theorem \ref{ksz}; i.e., $\left\Vert T\right\Vert
_{2}\leq n^{\frac{1}{2}}$. Note that when we consider the same multilinear
form in $\ell_{1}$, i.e., $T:\ell_{1}^{n}\times\cdots\times\ell_{1}%
^{n}\longrightarrow\mathbb{K}$, then $\left\Vert T\right\Vert _{1}=1.$ Hence,
for intermediate spaces $\ell_{q}$, by the Riesz-Thorin Theorem we have
$T:\ell_{q}^{n}\times\cdots\times\ell_{q}^{n}\longrightarrow\mathbb{K}$ \ with
$\left\Vert T\right\Vert _{q}\leq\left\Vert T\right\Vert _{2}^{\theta
}.\left\Vert T\right\Vert _{1}^{1-\theta}$ for $0\leq\theta\leq1$ and
\[
\theta=\frac{2q-2}{q}.
\]
Hence%
\[
\left\Vert T\right\Vert _{q}\leq n^{1-\frac{1}{q}}.
\]

So, as a matter of fact the  Kahane-Salem-Zygmund inequality can be stated in a more precise form as follows:

\begin{theorem} Let $n,m$  be positive integers and $p\geq1$. Then there exists a $m$-linear form $A:\ell_{p}%
^{n}\times\cdots\times\ell_{p}^{n}\longrightarrow\mathbb{K}$ \ of the form
\[
A\left(  x^{\left(  1\right)  },\ldots,x^{\left(  m\right)  }\right)  =%
{\textstyle\sum\limits_{i_{1},\ldots,i_{m}=1}^{n}}
\delta_{i_{1}...i_{m}}x_{i_{1}}^{\left(  1\right)  }\cdots x_{i_{m}}^{\left(
m\right)  },
\]
with $\delta_{i_{1}...i_{m}}\in\{-1,1\},$ such that
\[
\left\Vert A\right\Vert \leq C_{m}n^{\max\left\{  m\left(  \frac{1}{2}%
-\frac{1}{p}\right)  +\frac{1}{2}  ,1-\frac{1}{p}\right\}  }.
\]

\end{theorem}

We can now state a polynomial version of the Kahane-Salem-Zygmund inequality that is probably known but we
were not able to find in the literature:

\begin{theorem}
\label{poly} Let $n,m$ be positive integers and $p\geq1$. Then there exists a
$m$-homogeneous polynomial $P:\ell_{p}\longrightarrow\mathbb{K}$ \ of the
form
\begin{equation}
P\left(  x\right)  =%
{\textstyle\sum\limits_{i_{1},\ldots,i_{m}=1}^{\infty}}
\delta_{i_{1},\ldots,i_{m}}x_{i_{1}}\cdots x_{i_{m}} \label{7t}%
\end{equation}
with $\delta_{i_{1},\ldots,i_{m}}\in\left\{  0,1,-1\right\}  $ such that
\[
\left\Vert P\right\Vert \leq C_{m}n^{\max\left\{  m\left(  \frac{1}{2}%
-\frac{1}{p}\right)  +\frac{1}{2}  ,1-\frac{1}{p}\right\}  }%
\]
and%
\[
card\left\{  \delta_{i_{1},\ldots,i_{m}}\neq0\right\}  =n^{m}.
\]

\end{theorem}

\begin{proof}
By the previous theorem, there exists a $m$-linear
form
\[
T_{n}:\ell_{p}^{n}\times\cdots\times\ell_{p}^{n}\longrightarrow\mathbb{K}%
\]
of the form
\[
T_{n}\left(  x^{\left(  1\right)  },\ldots,x^{\left(  m\right)  }\right)  =%
{\textstyle\sum\limits_{i_{1},\ldots,i_{m}=1}^{n}}
\pm x_{i_{1}}^{\left(  1\right)  }\cdots x_{i_{m}}^{\left(  m\right)  }%
\]
such that
\[
\left\Vert T_{n}\right\Vert \leq C_{m}.n^{\max\left\{  m\left(  \frac{1}%
{2}-\frac{1}{p}\right)  +\frac{1}{2}  ,1-\frac{1}{p}\right\}  }.
\]

Consider $P:\ell_{p}\longrightarrow\mathbb{K}$ defined by
\[
P\left(  x\right)  =T_{n}\left(  \left(  x_{j}\right)  _{j\in\mathbb{N}_{1}%
},\ldots,\left(  x_{j}\right)  _{j\in\mathbb{N}_{m}}\right)  ,
\]
where%
\[
\mathbb{N}=\mathbb{N}_{1}%
{\textstyle\bigcup}
\cdots%
{\textstyle\bigcup}
\mathbb{N}_{m}%
\]
is a disjoint union, with $card\left(  \mathbb{N}_{j}\right)  =card\left(
\mathbb{N}\right)  $, for all $j=1,\ldots,m$, and%
\[
\left(  x_{j}\right)  _{j\in\mathbb{N}_{k}}:=\left\{
\begin{array}
[c]{c}%
x_{j},\text{if }j\in\mathbb{N}_{k}\\
0,\text{ otherwise.}%
\end{array}
\right.
\]
Note that $P$ is a $m$-homogeneous polynomial as in (\ref{7t}) and, moreover,
$\left\Vert P\right\Vert \leq\left\Vert T_{n}\right\Vert .$ By the estimate
obtained from the multilinear version, we have
\[
\left\Vert P\right\Vert \leq C_{m}n^{\max\left\{  m\left(  \frac{1}{2}%
-\frac{1}{p}\right)  +\frac{1}{2}  ,1-\frac{1}{p}\right\}  }.
\]

\end{proof}

A similar argument gives us a version of the Kahane--Salem--Zygmud inequality
for multipolynomials that will be useful in the next section.

\begin{theorem}
\label{mult}Let $n,$ $m,n_{1},\ldots,n_{m}$ be positive integers and $p\geq1$.
Then, there exists a $\left(  n_{1},\ldots,n_{m}\right)  $-homogeneous
polynomial $Q:\ell_{p}\times\cdots\times\ell_{p}\longrightarrow\mathbb{K}$
\ of the form
\[
Q\left(  x^{\left(  1\right)  },\ldots,x^{\left(  m\right)  }\right)  =%
{\textstyle\sum\limits_{i_{1},\ldots,i_{m}=1}^{\infty}}
\delta_{i_{1},\ldots,i_{m}}x_{i_{1}}^{\left(  1\right)  }\cdots x_{i_{m}%
}^{\left(  m\right)  },
\]
with $\delta_{i_{1},\ldots,i_{m}}\in\left\{  0,1,-1\right\}  $ and
$card\left\{  \left(  i_{1},...,i_{m}\right)  :\delta_{i_{1},\ldots,i_{m}}%
\neq0\right\}  =n^{M}$ such that
\[
\left\Vert Q\right\Vert \leq C_{n_{1}+\cdots+n_{m}}.n^{\max\left\{  M\left(
\frac{1}{2}-\frac{1}{p}\right)  +\frac{1}{2}  ,1-\frac{1}{p}\right\}  }.
\]

\end{theorem}

\begin{proof}
For each $n_{i}\geq1$ and $i=1,\ldots,m$, by the multilinear
Kahane-Salem-Zygmund inequality, there exists an $M$-linear form
\[
A:\overset{n_{1}}{\overbrace{\ell_{p}\times\cdots\times\ell_{p}}}\times
\cdots\times\overset{n_{m}}{\overbrace{\ell_{p}\times\cdots\times\ell_{p}}%
}\longrightarrow\mathbb{K}%
\]
of the form
\[
A\left(  x^{\left(  1\right)  },\ldots,x^{(M)}\right)  =%
{\textstyle\sum\limits_{i_{1},\ldots,i_{M}=1}^{n}}
\pm x_{i_{1}}^{\left(  1\right)  }\cdots x_{i_{M}}^{\left(  M\right)  }%
\]
such that
\[
\left\Vert A\right\Vert \leq C_{M}.n^{\max\left\{  M\left(  \frac{1}{2}%
-\frac{1}{p}\right)  +\frac{1}{2}  ,1-\frac{1}{p}\right\}  },
\]
where  $M=%
{\textstyle\sum\nolimits_{j=1}^{m}}
n_{j.}.$

Hence, consider $Q:\overset{m}{\overbrace{\ell_{p}\times\cdots\times\ell_{p}}%
}\longrightarrow\mathbb{K}$ given by%
\[
Q\left(  x^{\left(  1\right)  },\ldots,x^{\left(  m\right)  }\right)
=A\left(  \left(  x_{j}^{(1)}\right)  _{j\in\mathbb{N}_{1}^{(1)}}%
,\ldots,\left(  x_{j}^{(1)}\right)  _{j\in\mathbb{N}_{n_{1}}^{\left(
1\right)  }},...,\left(  x_{j}^{(m)}\right)  _{j\in\mathbb{N}_{1}^{(m)}%
},\ldots,\left(  x_{j}^{(m)}\right)  _{j\in\mathbb{N}_{n_{m}}^{\left(
m\right)  }}\right)
\]
with
\begin{align*}
\mathbb{N}  &  =\mathbb{N}_{1}^{\left(  1\right)  }%
{\textstyle\bigcup}
\cdots%
{\textstyle\bigcup}
\mathbb{N}_{n_{1}}^{\left(  1\right)  }\\
&  \vdots\\
\mathbb{N}  &  =\mathbb{N}_{1}^{\left(  m\right)  }%
{\textstyle\bigcup}
\cdots%
{\textstyle\bigcup}
\mathbb{N}_{n_{m}}^{\left(  m\right)  }%
\end{align*}
disjoint unions, with $card\left(  \mathbb{N}_{k}^{\left(  i\right)  }\right)
=card\left(  \mathbb{N}\right)  $, for all $i=1,\ldots,m$ and all
$k=1,\ldots,n_{i}.$ Note that $Q$ is a $\left(  n_{1},\ldots,n_{m}\right)
$-homogeneous polynomial and, moreover,
\[
\left\Vert Q\right\Vert \leq\left\Vert A\right\Vert \leq C_{M}.n^{\max\left\{
M\left(  \frac{1}{2}-\frac{1}{p}\right)  +\frac{1}{2}  ,1-\frac{1}{p}\right\}  }.
\]

\end{proof}

\section{Hardy--Littlewood inequality for multipolynomials: first case}
In this section we present a version of the  Hardy-Littlewood inequality for homogeneous multipolynomials.
\newline Note that, a $\left(  n_{1},\ldots,n_{m}\right)  $-homogeneous polynomial
$P:\ell_{p}\times\cdots\times\ell_{p}\longrightarrow\mathbb{K}$ can always be
written as%
\[
P\left(  x^{\left(  1\right)  },\ldots,x^{\left(  m\right)  }\right)
=\sum\limits_{\left\vert \alpha^{\left(  1\right)  }\right\vert =n_{1}%
,\ldots,\left\vert \alpha^{\left(  m\right)  }\right\vert =n_{m}}%
c_{\alpha^{\left(  1\right)  }\ldots\alpha^{\left(  m\right)  }}\left(
P\right)  \left(  x^{\left(  1\right)  }\right)  ^{\alpha^{\left(  1\right)
}}\cdots\left(  x^{\left(  m\right)  }\right)  ^{\alpha^{\left(  m\right)  }%
},
\]
where $\left(  \alpha_{j}^{\left(  i\right)  }\right)  _{j=1}^{\infty}$ is a
sequence in $\mathbb{N}\cup\left\{  0\right\}  ,$
\[
\left\vert \alpha^{\left(  i\right)  }\right\vert =%
{\textstyle\sum\limits_{j=1}^{\infty}}
\alpha_{j}^{\left(  i\right)  }%
\]
and
\[
\left(  x^{\left(  i\right)  }\right)  ^{\alpha^{\left(  i\right)  }}=%
{\textstyle\prod\limits_{j=1}^{\infty}}
\left(  x_{j}^{\left(  i\right)  }\right)  ^{\alpha_{j}^{\left(  i\right)  }%
},
\]
for $i=1,\ldots,m$. Now we prove our main result:

\begin{theorem}
\label{multh} Let $m,n_{1}%
,\ldots,n_{m}$ be positive integers, $s\in(0,+\infty)$ and $p\geq2\left(
{\textstyle\sum\nolimits_{j=1}^{m}}
n_{j}\right)  $. The following assertions are equivalent:

$(i)$ There is a constant $C_{n_{1},\ldots,n_{m},s}^{\mathbb{K}}\geq1$ such
that
\[
\left(  \sum\limits_{\left\vert \alpha^{\left(  1\right)  }\right\vert
=n_{1},\ldots,\left\vert \alpha^{\left(  m\right)  }\right\vert =n_{m}%
}\left\vert c_{\alpha^{\left(  1\right)  }\ldots\alpha^{\left(  m\right)  }%
}\left(  P\right)  \right\vert ^{s}\right)  ^{\frac{1}{s}}\leq C_{n_{1}%
,\ldots,n_{m},s}^{\mathbb{K}}\text{ }\left\Vert P\right\Vert
\]
for all $\left(  n_{1},\ldots,n_{m}\right)  $-homogeneous polynomial
$P:\ell_{p}\times\cdots\times\ell_{p}\longrightarrow\mathbb{K}.$

$(ii)$%
\[
s\geq\frac{2\left(
{\textstyle\sum\nolimits_{j=1}^{m}}
n_{j}\right)  p}{\left(
{\textstyle\sum\nolimits_{j=1}^{m}}
n_{j}\right)  p+p-2\left(
{\textstyle\sum\nolimits_{j=1}^{m}}
n_{j}\right)  }.
\]

\end{theorem}

\begin{proof}
$(ii)\Rightarrow(i):$ It suffices to prove the assertion for
\[
s_{0}=\frac{2\left(
{\textstyle\sum\nolimits_{j=1}^{m}}
n_{j}\right)  p}{\left(
{\textstyle\sum\nolimits_{j=1}^{m}}
n_{j}\right)  p+p-2\left(
{\textstyle\sum\nolimits_{j=1}^{m}}
n_{j}\right)  }.
\]
Indeed, consider the $\left(  n_{1}+\cdots+n_{m}\right)  $-homogeneous
polynomial $\mathbf{Q}:\ell_{p}\longrightarrow\mathbb{K}$ given by
\[
\mathbf{Q}\left(  z\right)  :=P\left(  \left(  z_{j}\right)  _{j\in
\mathbb{N}_{1}},\ldots,\left(  z_{j}\right)  _{j\in\mathbb{N}_{m}}\right)  ,
\]
being $\mathbb{N}=\mathbb{N}_{1}%
{\textstyle\bigcup}
\cdots%
{\textstyle\bigcup}
\mathbb{N}_{m}$ is a disjoint union with $card\left(  \mathbb{N}_{j}\right)  =card\left(
\mathbb{N}\right)  $, for $j=1,\ldots,m$. Moreover, note that
\[
\left\Vert \mathbf{Q}\right\Vert \leq\left\Vert P\right\Vert
\]
and
\[%
{\textstyle\sum\limits_{\left\vert \beta\right\vert =n_{1}+\cdots+n_{m}}}
\left\vert c_{\beta}\left(  \mathbf{Q}\right)  \right\vert ^{s}=\sum
\limits_{\left\vert \alpha^{\left(  1\right)  }\right\vert =n_{1}%
,\ldots,\left\vert \alpha^{\left(  m\right)  }\right\vert =n_{m}}\left\vert
c_{\alpha^{\left(  1\right)  }\ldots\alpha^{\left(  m\right)  }}\left(
P\right)  \right\vert ^{s}%
\]
for all $s.$ Hence, by the polynomial Hardy-Littlewood inequality in
(\ref{phl}), there exists a constant $C_{\mathbb{K}\text{,}n_{1}+\cdots
+n_{m},s_{0}}^{pol}\geq1$ such that
\begin{align*}
\left(  \sum\limits_{\left\vert \alpha^{\left(  1\right)  }\right\vert
=n_{1},\ldots,\left\vert \alpha^{\left(  m\right)  }\right\vert =n_{m}%
}\left\vert c_{\alpha^{\left(  1\right)  }\ldots\alpha^{\left(  m\right)  }%
}\left(  P\right)  \right\vert ^{s_{0}}\right)  ^{\frac{1}{s_{0}}}  &
=\left(
{\textstyle\sum\limits_{\left\vert \beta\right\vert =n_{1}+\cdots+n_{m}}}
\left\vert c_{\beta}\left(  \mathbf{Q}\right)  \right\vert ^{s_{0}}\right)
^{\frac{1}{s_{0}}}\\
&  \leq C_{\mathbb{K}\text{,}n_{1}+\cdots+n_{m},s_{0}}^{pol}\text{ }\left\Vert
\mathbf{Q}\right\Vert \\
&  \leq C_{\mathbb{K}\text{,}n_{1}+\cdots+n_{m},s_{0}}^{pol}\text{ }\left\Vert
P\right\Vert 
\end{align*}
as we wanted.

$(i)\Rightarrow(ii):$ Let $M=%
{\textstyle\sum\nolimits_{j=1}^{m}}
n_{j.}.$ For all $n$, let
\begin{align*}
Q_{n}  &  :\ell_{p}\times\cdots\times\ell_{p}\longrightarrow\mathbb{K}\\
Q_{n}\left(  x^{\left(  1\right)  },\ldots,x^{\left(  M\right)  }\right)   &
=%
{\textstyle\sum\limits_{i_{1},\ldots,i_{M}=1}^{\infty}}
\delta_{i_{1},\ldots,i_{M}}x_{i_{1}}^{\left(  1\right)  }\cdots x_{i_{M}%
}^{\left(  M\right)  }%
\end{align*}
be the $\left(  n_{1},\ldots,n_{m}\right)  $-homogeneous polynomial given by
the multipolynomial Kahane-Salem-Zygmund inequality, i.e.,%
\[
\left\{
\begin{array}
[c]{c}%
\left\Vert Q_{n}\right\Vert \leq C_{M}.n^{\max\left\{  M\left(  \frac{1}%
{2}-\frac{1}{p}\right)  +\frac{1}{2}  ,1-\frac{1}{p}\right\}  }\\
\delta_{i_{1},\ldots,i_{M}}\in\left\{  0,1,-1\right\} \\
\left\{  \left(  i_{1},...,i_{M}\right)  :\delta_{i_{1},\ldots,i_{M}}%
\neq0\right\}  =n^{M}.
\end{array}
\right.
\]

Define $P_{n}:\ell_{p}\times\cdots\times\ell_{p}\longrightarrow\mathbb{K}$
given by%
\begin{align*}
&  P_{n}\left(  x^{\left(  1\right)  },\ldots,x^{\left(  m\right)  }\right) \\
&  =Q_{n}\left(  \underbrace{\left(  x_{j}^{\left(  1\right)  }\right)
_{j\in\mathbb{N}_{1}^{\left(  1\right)  }},\ldots,\left(  x_{j}^{\left(
1\right)  }\right)  _{j\in\mathbb{N}_{n_{1}}^{\left(  1\right)  }}}%
,\ldots,\underbrace{\left(  x_{j}^{\left(  m\right)  }\right)  _{j\in
\mathbb{N}_{1}^{\left(  m\right)  }},\ldots,\left(  x_{j}^{\left(  m\right)
}\right)  _{j\in\mathbb{N}_{n_{m}}^{\left(  m\right)  }}}\right)  .\\
&  \text{ \ \ \ \ \ \ \ \ \ \ \ \ \ \ \ \ \ \ \ \ \ \ \ \ \ \ \ }n_{1}\text{
\ \ \ \ \ \ \ \ \ \ \ \ \ \ \ \ \ \ \ \ \ \ \ \ \ \ \ \ \ \ \ \ \ \ \ \ \ }%
n_{m}%
\end{align*}
Note that $P_{n}$ is a $\left(  n_{1},\ldots,n_{m}\right)  $-homogeneous
polynomial and $\left\Vert P_{n}\right\Vert \leq\left\Vert Q_{n}\right\Vert $.
Moreover, as the polynomial above has exactly $n^{M}$ non zero monomials we
have
\[%
{\textstyle\sum\limits_{\left\vert \alpha\right\vert =n_{1}+\cdots+n_{m}}}
\left\vert c_{\alpha}\left(  P_{n}\right)  \right\vert ^{s}=%
{\textstyle\sum\limits_{i_{1},\ldots,i_{M}=1}^{\infty}}
\left\vert Q_{n}\left(  e_{i_{1}},\ldots,e_{i_{M}}\right)  \right\vert
^{s}=n^{M}.
\]
Again, by the multipolynomial Kahane-Salem-Zygmund inequality (Theorem
\ref{mult}), we have
\[
\left\Vert P_{n}\right\Vert \leq\left\Vert Q_{n}\right\Vert \leq
C_{n_{1}+\cdots+n_{m}}.n^{M\left(  \frac{1}{2}-\frac{1}{p}\right)  +\frac
{1}{2}}%
\]
and thus
\[
\left(  n^{M}\right)  ^{\frac{1}{s}}\leq C_{n_{1}+\cdots+n_{m}}.n^{\frac{1}%
{2}+M\left(  \frac{1}{2}-\frac{1}{p}\right)  }.
\]
Since $n$ is arbitrary, we get
\[
\frac{M}{s}\leq\frac{Mp+p-2M}{2p}%
\]
and thus $s\geq\frac{2Mp}{Mp+p-2M}.$
\end{proof}

\section{Hardy-Littlewood inequality for multipolynomials: second case}

Using similar arguments from the above section, the main goal of this section
is to prove a new version of the multipolynomial Hardy-Littlewood inequality for
the case $M<p\leq2M$, i.e, we extend \ref{hlpo}  to the  multipolynomials.

\begin{theorem}
Let $m,n_{1},\ldots,n_{m}$ positive integers, $s\in\left(  0,+\infty\right)  $ and
	$M=%
	{\textstyle\sum\limits_{j=1}^{m}}
	n_{j}.$ The following assertions are equivalent:
	
	$\left(  i\right)  $ There is a constant $C_{n_{1},\ldots,n_{m},s}%
	^{\mathbb{K}}\geq1$ such that
	\[
	\left(
	{\textstyle\sum\limits_{\left\vert \alpha^{\left(  1\right)  }\right\vert
			=n_{1},\ldots,\left\vert \alpha^{\left(  m\right)  }\right\vert =n_{m}}}
	\left\vert c_{\alpha^{\left(  1\right)  }\ldots\alpha^{\left(  m\right)  }%
	}\left(  P\right)  \right\vert ^{s}\right)  ^{\frac{1}{s}}\leq C_{n_{1}%
	,\ldots,n_{m},s}^{\mathbb{K}}\text{ }\left\Vert P\right\Vert
\]
for all $\left(  n_{1},\ldots,n_{m}\right)  $-homogeneous polynomial
$P:\ell_{p}\times\cdots\times\ell_{p}\longrightarrow\mathbb{K}$.

$\left(  ii\right)  $ $s\geq\frac{p}{p-M}$
\end{theorem}

\begin{proof}
	$\left(  ii\right)  \Rightarrow\left(  i\right)  :$ We prove the assertion for
	$s_{0}=\frac{p}{p-M}.$ In fact, Consider the $\left(  n_{1}+\cdots
	+n_{m}\right)  $-homogeneous polynomial $Q:\ell_{p}\longrightarrow\mathbb{K}$
	given by%
	\[
	Q\left(  z\right)  :=P\left(  \left(  z_{j}\right)  _{j\in\mathbb{N}_{1}%
	},\ldots,\left(  z_{j}\right)  _{j\in\mathbb{N}_{m}}\right)
	\]
	being $\mathbb{N=N}_{1}%
	{\textstyle\bigcup}
	\cdots%
	{\textstyle\bigcup}
	\mathbb{N}_{m}$ is a disjoint union with $card\left(  \mathbb{N}_{j}\right)
	=card\left(  \mathbb{N}\right)  $, for $j=1,\ldots,m$ and
	\[
	\left(  x_{j}\right)  _{j\in\mathbb{N}_{k}}:=\left\{
	\begin{array}
	[c]{c}%
	x_{j},\text{ if }j\in\mathbb{N}_{k}\\
	0,\text{ otherwise.}%
	\end{array}
	\right.
	\]
	Note that $Q$ is a $M$- homogeneous polynomial. Moreover, $\left\Vert
	Q\right\Vert \leq\left\Vert P\right\Vert $ and%
	\[%
	{\textstyle\sum\limits_{\left\vert \beta\right\vert =n_{1}+\cdots+n_{m}}}
	\left\vert c_{\beta}\left(  Q\right)  \right\vert ^{s}=%
	{\textstyle\sum\limits_{\left\vert \alpha^{\left(  1\right)  }\right\vert
			=n_{1},\ldots,\left\vert \alpha^{\left(  m\right)  }\right\vert =n_{m}}}
	\left\vert c_{\alpha^{\left(  1\right)  }\ldots\alpha^{\left(  m\right)  }%
	}\left(  P\right)  \right\vert ^{s}%
	\]
	for all $s$. Hence, by polynomial Hardy-Littlewood inequality for $m<p\leq2m$,
	(\ref{hlpo}), (note that $m\leq M),$ there exists a constant
	$C_{\mathbb{K}\text{,},n_{1}+\cdots+n_{m},s_{0}}^{pol}\geq1$ such that
	\begin{align*}
	\left(
	{\textstyle\sum\limits_{\left\vert \alpha^{\left(  1\right)  }\right\vert
			=n_{1},\ldots,\left\vert \alpha^{\left(  m\right)  }\right\vert =n_{m}}}
	\left\vert c_{\alpha^{\left(  1\right)  }\ldots\alpha^{\left(  m\right)  }%
	}\left(  P\right)  \right\vert ^{s_{0}}\right)  ^{\frac{1}{s_{0}}}  &
	=\left(
	{\textstyle\sum\limits_{\left\vert \beta\right\vert =n_{1}+\cdots+n_{m}}}
	\left\vert c_{\beta}\left(  Q\right)  \right\vert ^{s_{0}}\right)  ^{\frac
		{1}{s_{0}}}\\
	&  \leq C_{\mathbb{K}\text{,},n_{1}+\cdots+n_{m},s_{0}}^{pol}\text{
	}\left\Vert Q\right\Vert \\
	&  \leq C_{\mathbb{K}\text{,},n_{1}+\cdots+n_{m},s_{0}}^{pol}\text{
	}\left\Vert P\right\Vert .
	\end{align*}

	$\left(  i\right)  \Rightarrow\left(  ii\right)  $ Let us assume now that $s$
	is such that for every $\left(  n_{1},\ldots,n_{m}\right)  $-homogeneous
	polynomial $P:\ell_{p}\times\cdots\times\ell_{p}\longrightarrow\mathbb{K}$ we
	have%
	\[
	\left(
	{\textstyle\sum\limits_{\left\vert \alpha^{\left(  1\right)  }\right\vert
			=n_{1},\ldots,\left\vert \alpha^{\left(  m\right)  }\right\vert =n_{m}}}
	\left\vert c_{\alpha^{\left(  1\right)  }\ldots\alpha^{\left(  m\right)  }%
	}\left(  P\right)  \right\vert ^{s}\right)  ^{\frac{1}{s}}\leq C\left\Vert
	P\right\Vert .
	\]
	Given a positive integer $n$, define the $M$-linear operator $T_{n}:$
	$\ell_{p}\times\cdots\times\ell_{p}\longrightarrow\mathbb{K}$ by%
	\[
	T_{n}\left(  x^{\left(  1\right)  },\ldots,x^{\left(  M\right)  }\right)  =%
	{\textstyle\sum\limits_{i=1}^{n}}
	x_{i}^{\left(  1\right)  }\ldots x_{i}^{\left(  M\right)  }%
	\]
	with $M=%
	{\textstyle\sum\limits_{j=1}^{m}}
	n_{j}.$ Define $P_{n}:\overset{m}{\overbrace{\ell_{p}\times\cdots\times
			\ell_{p}}}\longrightarrow\mathbb{K}$ by%
	\begin{align*}
	&  P_{n}\left(  x^{\left(  1\right)  },\ldots,x^{\left(  m\right)  }\right)
	\\
	&  =T_{n}\left(  \left(  x_{j}^{\left(  1\right)  }\right)  _{j\in
		\mathbb{N}_{1}^{\left(  1\right)  }},\ldots,\left(  x_{j}^{\left(  1\right)
	}\right)  _{j\in\mathbb{N}_{n_{1}}^{\left(  1\right)  }},\ldots,\left(
	x_{j}^{\left(  m\right)  }\right)  _{j\in\mathbb{N}_{1}^{\left(  m\right)  }%
	},\ldots,\left(  x_{j}^{\left(  m\right)  }\right)  _{j\in\mathbb{N}_{n_{m}%
	}^{\left(  m\right)  }}\right)
\end{align*}
where
\begin{align*}
\mathbb{N} &  \mathbb{=N}_{1}^{\left(  1\right)  }%
{\textstyle\bigcup}
\cdots%
{\textstyle\bigcup}
\mathbb{N}_{n_{1}}^{\left(  1\right)  }\\
&  \vdots\\
\mathbb{N} &  \mathbb{=N}_{1}^{\left(  m\right)  }%
{\textstyle\bigcup}
\cdots%
{\textstyle\bigcup}
\mathbb{N}_{n_{m}}^{\left(  m\right)  }%
\end{align*}
are disjoint unions with $card\left(  \mathbb{N}_{k}^{\left(  i\right)
}\right)  =card\left(  \mathbb{N}\right)  $, for $i=1,\ldots,m$ and
$k=1,\ldots,n_{i}.$ Note that $P_{n}$ is an $\left(  n_{1},\ldots
,n_{m}\right)  $-homogeneous polynomial and $\left\Vert P_{n}\right\Vert
\leq\left\Vert T_{n}\right\Vert $. Moreover, using H\"{o}lder's inequality it
is easily seen that $\left\Vert T_{n}\right\Vert \leq n^{\frac{p-M}{p}}$.  By hypothesis,%
\[
n^{\frac{1}{s}}\leq C\left\Vert P_{n}\right\Vert \leq C\left\Vert
T_{n}\right\Vert \leq Cn^{\frac{p-M}{p}}.
\]
Since $n$ is arbitrary, we get%
\[
\frac{1}{s}\leq\frac{p-M}{p}%
\]
and thus $s\geq\frac{p}{p-M}.$
\end{proof}

\end{document}